\documentclass[11pt]{article}    % Specifies the document style.
\usepackage{amsmath,amssymb,amsthm,amsfonts,natbib}

\title{Geometry of diagonal-effect models for contingency tables}

\author{Cristiano Bocci \\ Department of Mathematics and Computer Science, University of Siena \\
Pian dei Mantellini, 44 \\ 53100 Siena, Italy \\
\tt{bocci24@unisi.it}
\and Enrico Carlini \\ Dipartimento di Matematica, Politecnico di Torino \\
Corso Duca degli Abruzzi, 24 \\
10129 Torino, Italy \\
\tt{enrico.carlini@polito.it}
\and Fabio Rapallo \\ Department DISTA, University of Eastern Piedmont \\
Viale Teresa Michel, 11 \\
15121 Alessandria, Italy \\ \tt{fabio.rapallo@mfn.unipmn.it} }

\begin{document}

\newtheorem{theorem}{Theorem}[section]
\newtheorem{definition}[theorem]{Definition}
\newtheorem{proposition}[theorem]{Proposition}
\newtheorem{lemma}[theorem]{Lemma}
\newtheorem{corollary}[theorem]{Corollary}
\newtheorem{remark}[theorem]{Remark}
\newtheorem{example}[theorem]{Example}

\maketitle                 % Produces the title.

\begin{abstract}
In this work we study several types of diagonal-effect models for
two-way contingency tables in the framework of Algebraic
Statistics. We use both toric models and mixture models to encode
the different behavior of the diagonal cells. We compute the
invariants of these models and we explore their geometrical
structure.
\bigskip

\end{abstract}

\section{Introduction} \label{intro}

A probability distribution on a finite sample space ${\mathcal X}$
with $k$ elements is a normalized vector of $k$ non-negative real
numbers. Thus, the most general probability model is the simplex
\begin{equation} \label{gen-simplex}
\Delta = \left\{(p_{1}, \ldots, p_{k}) \ : \ p_{i} \geq 0 \ , \
\sum_{i=1}^k p_{i} = 1   \right\} \, .
\end{equation}
A statistical model ${\mathcal M}$ is therefore a subset of
$\Delta$.

A classical example of finite sample space is the case of two-way
contingency tables, where the sample space is usually written as a
cartesian product of the form ${\mathcal X}=\{1, \ldots , I\}
\times \{1 , \ldots, J \}$. We will consider this case extensively
in the next sections.

When ${\mathcal M}$ is defined through algebraic equations, the
model ${\mathcal M}$ is said to be an algebraic model. In such a
case, algebraic and geometric techniques are useful to study the
structure of the model and many statistical quantities such as
sufficient statistics and maximum likelihood estimators. In recent
literature this approach is known as ``Algebraic Statistics''. For
a survey on this field the reader can refer to
\cite{pistone|riccomagno|wynn:01} and \cite{pachter|sturmfels:05}.

With this point of view, a statistical model is defined as the set
of points in $\Delta$ where certain polynomials $f_1(p_1, \ldots,
p_k), \ldots, f_\ell(p_1, \ldots , p_k)$ vanish. Notice that the
non-negativity and normalization conditions increase the
complexity of the geometrical study. In fact, Algebraic Geometry
usually works in complex projective spaces, see e.g.
\cite{harris:92}, while in Algebraic Statistics we have to
consider a real affine variety and we must intersect the variety
with the simplex.

A description of the problems raised by non-negativity and
normalization are described for instance in
\cite{pistone|riccomagno|wynn:01ams} and
\cite{geiger|meek|sturmfels:06}, while the use of Algebraic
Statistics in various applications is presented in
\cite{riccomagno:09}.

The difference between models with positive probabilities and
models with non-negative probabilities has been deeply studied in
Algebraic Statistics. When only positive probabilities are
involved usually one takes the log-probabilities $\log(p_1),
\ldots, \log(p_k)$. Hence, many statistical models are defined by
linear equations in the log-probabilities. The most widely used
models for contingency tables are defined in this way and are
called log-linear models, see \cite{agresti:02}. The use of
polynomial algebra instead of linear algebra has led to the study
of models with non-negative probabilities. Therefore, the new
class of toric models has been introduced and many geometric
properties have been related to statistical properties, see e.g.
\cite{geiger|meek|sturmfels:06}. Toric models generalize the
log-linear models to include models with structural zeros, see
e.g. \cite{rapallo:07}. Exact inference on toric model can be done
through MCMC methods based on Markov bases. The difference between
positivity and non-negativity is also a major issue in the
computation of Markov bases, see \cite{diaconis|sturmfels:98} and
\cite{chen|dinwoodie|dobra|huber:05} where the difference between
lattice bases and Markov bases is shown to be essential.

In this paper we consider the diagonal-effect models, i.e., models
encoding a special behavior of the diagonal cells of the table
with respect to the independence model. It is a class of
statistical models for square two-way contingency tables with a
wide range of applications, from social mobility analysis in
Psychometry to rater agreement analysis in medical and
pharmaceutical sciences, see e.g. \cite{agresti:92},
\cite{schuster:02}. Some results in Algebraic Statistics for this
kind of models have already been discussed in \cite{rapallo:05},
\cite{carlini|rapallo:08} and \cite{krampe|kuhnt:07}. Due to the
variety of the applications, this type of statistical models has
been approached in many different ways and several mathematical
definitions have been introduced, often to describe the same
objects. In this paper, we will concentrate especially on toric
models and mixture models.

The main aim of this paper is to study the geometric structure of
the diagonal-effect models, showing the differences between toric
models and mixture models. In particular, we compute the
invariants of these models. We recall that an invariant of a model
is a polynomial function vanishing in the points of the model, see
\cite{garcia|stillman|sturmfels:05}. We show that the toric and
mixture models differ not only on the boundary of the simplex but
also in its interior, also when the models have the same
invariants.

In Section \ref{recalls} we recall some basic definitions and
results on toric models, with special emphasis on the independence
model. In Section \ref{diageffsect} we define the diagonal-effect
models as both toric models and mixture models, we show that they
have the same invariants, and we describe their structure, while
in Section \ref{geo-desc} we study in more details their geometry.
Finally, in Section \ref{commdiagsect} we study a special class of
diagonal effect models which encodes a common behavior of the
diagonal cells, i.e., all diagonal cells give the same
contribution.

The results presented here also suggest future works on these
topics as such as: the comparison of two or more diagonal effect
models; the study of the geometry of more complex models, such as
diagonal models for multi-way tables or non-square tables; a
better understanding of the notion of maximum likelihood estimates
for this kind of models.

\section{Basic facts on toric models and independence} \label{recalls}

In this paper we consider a two-way contingency table as the joint
observed counts of two categorical random variables $X$ and $Y$.
Let us suppose that the random variable $X$ has $I$ levels, and
$Y$ has $J$ levels. Therefore, the sample space is the cartesian
product ${\mathcal X}=\{1, \ldots, I\} \times \{1, \ldots, J\}$
and the observed contingency table is a point $f \in {\mathbb
N}^{IJ}$.

A probability distribution for an $I \times J$ contingency table
is a matrix $P=(p_{i,j})$ such that $p_{i,j}= {\mathbb P}(X=i,
Y=j)$. Clearly the matrix $P$ is such that $p_{i,j} \geq 0$ for
all $i=1, \ldots, I$ and $j=1, \ldots J$, and $\sum_{i,j}
p_{i,j}=1$. In other words, the matrix $P$ is a point of the
closed simplex
\begin{equation} \label{closedsimplex}
\Delta = \left\{P=(p_{i,j}) \ : \ p_{i,j} \geq 0 \ , \ \sum_{i,j}
p_{i,j} = 1   \right\}\, .
\end{equation}

A statistical model is a subset of $\Delta$. In most cases, the
statistical model is defined through algebraic equations and the
model is said to be algebraic. A wide class of algebraic
statistical models is the class of toric models, see
\cite{pistone|riccomagno|wynn:01},
\cite{pistone|riccomagno|wynn:01ams} and \cite{rapallo:07}.

In a toric model, the raw probabilities of the cells are defined
in parametric form as power products through a map $\phi: {\mathbb
R}_{\geq 0}^s \rightarrow {\mathbb R}_{\geq 0}^{IJ}$:
\begin{equation} \label{partoric}
p_{i,j} = \prod_{h=1}^s \zeta_h^{A_{(i,j),h}}  \, .
\end{equation}
Therefore, the structure of the toric model is encoded in an $IJ
\times s$ non-negative integer matrix $A$ which extends $\phi$ to
a vector space homomorphism, see
\cite{pistone|riccomagno|wynn:01}, Chapter $6$. Notice that, in
the open simplex
\begin{equation} \label{opensimplex}
\Delta_{> 0} = \left\{P=(p_{i,j}) \ : \ p_{i,j} > 0 \ , \
\sum_{i,j} p_{i,j} = 1   \right\}
\end{equation}
the power product representation leads to a vector-space
representation by taking the log-probabilities. Moreover, it is
known that eliminating the $\zeta$ parameters from Equations in
\eqref{partoric} one obtains the toric ideal ${\mathcal I}_A$
associated to the statistical model. The ideal ${\mathcal I}_A$ is
a polynomial ideal in the ring ${\mathbb R}[p] = {\mathbb
R}[p_{1,1}, \ldots , p_{I,J}]$ generated by pure binomials. We
recall that a binomial $p^a-p^b$ is pure if
$\mathrm{gcd}(p^a,p^b)=1$. The notation $p^{a}-p^{b}$ is a vector
notation for $\prod_{i,j} p_{i,j}^{a_{i,j}} - \prod_{i,j}
p_{i,j}^{b_{i,j}}$.

A move for the toric model defined by the matrix $A$ is a table $m
\in {\mathbb Z}^{IJ}$ with integer entries such that $A^t(m)=0$.
The move $m$ is represented in the ring ${\mathbb R}[p]$ by the
pure binomial $p^{m+}-p^{m-}$, where $m+$ and $m-$ are the
positive and negative part of $m$.

A Markov basis for the statistical toric model defined by the
matrix $A$ is a finite set of tables $m_1, \ldots, m_\ell \in
{\mathbb Z}^{IJ}$ that connects any two contingency tables $f_1$
and $f_2$ in the same fiber, i.e. such that $A^t(f_1)=A^t(f_2)$,
with a path of elements of the fiber. The path is therefore formed
by tables of non-negative counts with constant image under $A^t$.

The relation between the notion of Markov basis and the toric
ideal ${\mathcal I}_A$ is given in the theorem below.

\begin{theorem}[\cite{diaconis|sturmfels:98}, Theorem 3.1] \label{theoldtheo}
The set of moves $\{m_1, \ldots, m_\ell\}$ is a Markov basis if
and only if the set $\{p^{m_1+}-p^{m_1-}, \ i=1 , \ldots , \ell
\}$ generates the ideal ${\mathcal I}_A$.
\end{theorem}

In many applications this theorem has been used in its ``if'' part
to deduce Markov bases from the computation of a system of
generators of a toric ideal, see e.g. \cite{rapallo:03} and
\cite{chen|dinwoodie|sullivant:06}. On the contrary, in the next
section we will make use of Theorem \ref{theoldtheo} in its ``only
if'' implication.

In this paper the independence model will play a special role. It
can be considered as the simplest toric model. The variables $X$
and $Y$ are independent if ${\mathbb P}(X=i, Y=j)={\mathbb
P}(X=i){\mathbb P}(Y=j)$, i.e., the joint distribution is the
product of the marginal distributions. The independence condition
can be written as:
\begin{equation} \label{indep1}
p_{i,j}=\zeta^{(r)}_i \zeta^{(c)}_j \ \ \ \mbox{ for all } i,j
\end{equation}
for suitable $\zeta^{(r)}_i$'s and $\zeta^{(c)}_j$'s. The
non-negativity constraint reflects into non-negativity of the
parameters. Namely, we suppose $\zeta^{(r)}_i \geq 0$ for all
$i=1, \ldots , I$, and $\zeta^{(c)}_j \geq 0$ for all $j=1, \ldots
J$. Using Equation \eqref{partoric}, the independence model is
then defined as the set
\begin{equation}
{\mathcal M} = \{P=(p_{i,j}) \ : \ p_{i,j}=\zeta^{(r)}_i
\zeta^{(c)}_j  \ , \ 1 \leq i \leq I, \ 1 \leq j \leq J \} \cap
\Delta \, ,
\end{equation}
for non-negative $\zeta^{(r)}_i$'s and $\zeta^{(c)}_j$'s.

Notice that Equation \eqref{indep1} implies that the matrix $P$
has rank $1$ and therefore a probability matrix $P$ in the
independence model must have all $2 \times 2$ minors equal to
zero. In formulae, it is therefore easy to write the independence
model in implicit form as:
\begin{equation} \label{minors}
{\mathcal M}' = \{P=(p_{i,j}) \ : \ p_{i,j}p_{k,h}-p_{i,h}p_{k,j}
= 0 \ , \ 1 \leq i < k \leq I, \ 1 \leq j < h \leq J \} \cap
\Delta \, .
\end{equation}
In \cite{diaconis|sturmfels:98}, the authors have studied this set
to find Markov bases for the independence model, while the
corresponding polynomial ideal has been considered in Algebraic
Geometry in the framework of determinantal ideals, see
\cite{hosten|sullivant:04}.

As the independence model is toric, Lemma $2$ in \cite{rapallo:07}
says that the model ${\mathcal M}$ in parametric form and the
corresponding model ${\mathcal M'}$ in implicit form coincide in
the open simplex $\Delta_{> 0}$.

\begin{proposition} \label{proprank1}
With the notation above, in $\Delta_{> 0}$ we have that ${\mathcal
M} = {\mathcal M}'$.
\end{proposition}
\begin{proof}
Using the same notation as above, consider the sets
\begin{equation*}
\{P=(p_{i,j}) \ : \ p_{i,j}=\zeta^{(r)}_i \zeta^{(c)}_j  \ , \ 1
\leq i \leq I, \ 1 \leq j \leq J \}
\end{equation*}
and \begin{equation*} \{P=(p_{i,j}) \ : \
p_{i,j}p_{k,h}-p_{i,h}p_{k,j} = 0 \ , \ 1 \leq i < k \leq I, \ 1
\leq j < h \leq J \} \, .
\end{equation*}
Taking the log-probabilities, both sets are defined as a linear
system and it is immediate to show that they define two vector
sub-spaces with the same dimension.
\end{proof}

It is known that ${\mathcal M}$ and ${\mathcal M}'$ are in general
different on the boundary $\Delta \setminus \Delta_{>0}$. A
complete description of this issue can be found in Section 4 of
\cite{rapallo:07}.

\section{Diagonal-effect models} \label{diageffsect}

As mentioned in the Introduction, diagonal-effect models for
square $I \times I$ tables can be defined in at least two ways. In
the field of toric models, one can define these models in monomial
form as follows.

\begin{definition} \label{diag-toric}
The diagonal-effect model ${\mathcal M}_1$ is defined as the set
of probability matrices $P \in \Delta$ such that:
\begin{equation} \label{toric-def-1}
p_{i,j} = \zeta^{(r)}_i \zeta^{(c)}_j \ \mbox{ for } \ i \ne j
\end{equation}
and
\begin{equation} \label{toric-def-2}
p_{i,j} = \zeta^{(r)}_i \zeta^{(c)}_j\zeta^{(\gamma)}_i \ \mbox{
for } \ i = j
\end{equation}
where $\zeta^{(r)}$, $\zeta^{(c)}$ and $\zeta^{(\gamma)}$ are
non-negative vectors with length $I$.
\end{definition}

In literature, such a model is also known as quasi-independence
model, see \cite{agresti:02}. As the model in Definition
\ref{diag-toric} is a toric model, it is relatively easy to find
the invariants. Eliminating the parameters $\zeta^{(r)}$,
$\zeta^{(c)}$ and $\zeta^{(\gamma)}$ one obtains the following
result.

\begin{proposition} \label{appl_DS}
The invariants of the model ${\mathcal M}_1$ are the binomials
\begin{equation} \label{binomi1}
p_{i,j}p_{i',j'} - p_{i,j'}p_{i',j}
\end{equation}
for $i$, $i'$, $j$, $j'$ all distinct, and
\begin{equation} \label{binomi2}
p_{i,i'}p_{i',i''}p_{i'',i} - p_{i,i''}p_{i'',i'}p_{i',i}
\end{equation}
for $i$, $i'$, $i''$ all distinct.
\end{proposition}
\begin{proof}
In \cite{aoki|takemura:05b}, it is shown that a minimal Markov
basis for the model ${\mathcal M}_1$ is formed by:
\begin{itemize}
\item The basic degree $2$ moves:
\begin{center}
\begin{tabular}{ccc}
    & $j$ & $j'$ \\
$i$ & $+1$ & $-1$ \\
$i'$ & $-1$ & $+1$
\end{tabular}
\end{center}
with $i$, $i'$, $j$, $j'$ all distinct, for $I \geq 4$;

\item The degree $3$ moves of the form:
\begin{center}
\begin{tabular}{cccc}
    & $i$ & $i'$ & $i''$ \\
$i$ & $0$ & $+1$ & $-1$ \\
$i'$ & $-1$ & $0$ & $+1$ \\
$i''$ & $+1$ & $-1$ & $0$
\end{tabular}
\end{center}
with $i$, $i'$, $i''$ all distinct, for $I \geq 3$.
\end{itemize}
Thus, using Theorem \ref{theoldtheo}, the binomials in Equations
\ref{binomi1} and \ref{binomi2} form a set of generators of the
toric ideal associated to the model ${\mathcal M}_1$.
\end{proof}

\begin{remark}
To study the geometry of the model with structural zeros on the
main diagonal it is enough to consider the variety defined by the
polynomials in Proposition \ref{appl_DS} and intersect it with the
hyperplanes $\{p_{i,i}=0\}$ for all $i$.
\end{remark}

In the framework of the mixture models, the diagonal-effect models
have an alternative definition as follows.

\begin{definition} \label{diag-mixt}
The diagonal-effect model ${\mathcal M}_2$ is defined as the set
of probability matrices $P$ such that
\begin{equation} \label{mixtdef}
P = \alpha cr^t + (1-\alpha) D
\end{equation}
where $r$ and $c$ are non-negative vectors with length $I$ and sum
$1$, $D={\mathrm{diag}}(d_1, \ldots, d_I)$ is a non-negative
diagonal matrix with sum $1$, and $\alpha \in [0,1]$.
\end{definition}

\begin{remark}
Notice that while in Definition \ref{diag-toric} the normalization
is applied once, in Definition \ref{diag-mixt} the normalization
is applied twice as we require that both $cr^t$ and $D$ are
probability matrices. This difference will be particularly
relevant in the study of the geometry of the models.
\end{remark}

First, we study the invariants and some geometrical properties of
these models, then we will give some results on their sufficient
statistics.

\begin{theorem}
The models ${\mathcal M}_1$ and ${\mathcal M}_2$ have the same
invariants.
\end{theorem}
\begin{proof}
Writing explicitly the polynomials in Equations
\eqref{toric-def-1} and \eqref{toric-def-2} it is easy to check
that each $\zeta^{(\gamma)}_i$ appears in only one polynomial. The
same for each $d_i$ in Equations \eqref{mixtdef}. Thus, following
Theorem 3.4.5 in \cite{kreuzer|robbiano:00}, such polynomials are
deleted when we eliminate the indeterminates
$\zeta^{(\gamma)}_i$'s and $d_i$'s.

As the remaining polynomials, corresponding to off-diagonal cells,
are the same in both models, the models ${\mathcal M}_1$ and
${\mathcal M}_2$ have the same invariants.
\end{proof}

In order to study in more details the connections between
${\mathcal M}_1$ and ${\mathcal M}_2$ we further investigate their
geometric structure. The non-negativity conditions imposed in the
definitions imply that ${\mathcal M}_1 \not= {\mathcal M}_2$ and
neither ${\mathcal M}_2 \subset {\mathcal M}_1$ nor ${\mathcal
M}_1 \subset {\mathcal M}_2$. We can show this by two easy
examples.

First, let $\zeta^{(r)}$ and $\zeta^{(c)}$ respectively the
vectors, of length $I$, $(\frac{1}{I}, \frac{1}{I}, \dots,
\frac{1}{I})$ and $(\frac{1}{I-1}, \frac{1}{I-1}, \dots,
\frac{1}{I-1})$ and define $\zeta^{(\gamma)}$ as the zero vector.
Thus, the probability table we obtain in toric form is:
\begin{equation*}
P=\begin{pmatrix}
0 & \frac{1}{I(I-1)} & \frac{1}{I(I-1)} & \dots &\frac{1}{I(I-1)}\\
\frac{1}{I(I-1)} & 0 & \frac{1}{I(I-1)} &  \dots &\frac{1}{I(I-1)}\\
\frac{1}{I(I-1)} & \frac{1}{I(I-1)} & 0 & \dots & \frac{1}{I(I-1)} \\
\vdots & \vdots & \vdots & \vdots & \vdots \\
\frac{1}{I(I-1)} & \frac{1}{I(I-1)} & \frac{1}{I(I-1)} & \ldots &
0
\end{pmatrix} \, .
\end{equation*}
Such probability matrix belongs to $\mathcal{M}_1$ by
constructions, while it does not belong to ${\mathcal M}_2$. In
fact, $p_{1,1}=0$ in Equation \eqref{mixtdef} would imply either
$\alpha=0$ (a contradiction, as $P$ is not a diagonal matrix), or
$\zeta^{(r)}_1=0$ (a contradiction, as $P$ has not the first row
with all $0$'s), or $\zeta^{(c)}_1=0$ (a contradiction, as $P$ has
not the first column with all $0$'s).

On the other hand, let $P$ be the diagonal matrix
\begin{equation*}
P=\begin{pmatrix}
\frac{1}{I} & 0 & 0 & \dots & 0\\
0 & \frac{1}{I} & 0 & \dots & 0\\
0 & 0 & \frac{1}{I} & \dots & 0 \\
\vdots & \vdots & \vdots & \vdots & \vdots \\
0 & 0 & 0 & \ldots & \frac{1}{I}
\end{pmatrix} \, .
\end{equation*}
Such probability matrix belongs to ${\mathcal M}_2$ by setting
$\alpha=0$ and $D=\mathrm{diag}(\frac{1}{I}, \ldots ,
\frac{1}{I})$, while it does not  belong to ${\mathcal M}_1$. To
prove this it is enough to note that $p_{1,2}=0$ would imply
either $\zeta^{(r)}_1=0$ (a contradiction, as the first row of $P$
is not zero), or $\zeta^{(c)}_2=0$ (a contradiction, as the second
column of $P$ is not zero).

Nevertheless, in the open simplex we can prove one of the
inclusions.

\begin{proposition} \label{prop-incl1}
In the open simplex $\Delta_{>0}$,
\begin{equation}
{\mathcal M}_2 \subset {\mathcal M}_1
\end{equation}
\end{proposition}
\begin{proof}
In fact, let us consider a probability table in ${\mathcal M}_2$,
given by $P = \alpha cr^t + (1-\alpha) D$. As $P \in \Delta_{>0}$,
$\alpha \not=0$, $r_i \not= 0$ for all $i=1, \ldots, I$ and $c_j
\not=0$ for all $j=1, \ldots, I$. Then we can describe $P$ as an
element of ${\mathcal M}_1$ in the following way. We define
$\zeta^{(r)}_i = r_i$ for all $i$ and $\zeta^{(c)}_j= \alpha c_j$,
for all $j$. After that, it is enough to find the diagonal
parameters by solving the equations
\begin{equation*}
\alpha r_i c_i \zeta^{(\gamma)}_i=\alpha r_i c_i + (1-\alpha) d_i
\end{equation*}
that is, as $\alpha\not= 0$, $r_i\not= 0$, and $c_i\not= 0$, we
have
\begin{equation*}
\zeta^{(\gamma)}_i=1+\frac{(1-\alpha) d_i}{\alpha r_i c_i} \, .
\end{equation*}
\end{proof}

Moreover, in the open simplex $\Delta_{>0}$, the inclusion in
Proposition \ref{prop-incl1} is strict. Let us analyze the
probability matrices in the difference ${\mathcal M}_1 \setminus
{\mathcal M}_2$.

Consider three vectors $\zeta^{(r)}=(\zeta^{(r)}_1, \dots,
\zeta^{(r)}_I)$, $\zeta^{(c)}=(\zeta^{(c)}_1, \dots,
\zeta^{(c)}_I)$ and $\zeta^{(\gamma)}=(\zeta^{(\gamma)}_1, \dots,
\zeta^{(\gamma)}_I)$. Using these vectors, we define the
probability table $P$ as in Definition \ref{diag-toric} and then
we normalize it, i.e. dividing by $N_T=\sum_{i\not=
j}\zeta^{(r)}_i\zeta^{(c)}_j+\sum_{i=j}\zeta^{(r)}_i\zeta^{(c)}_j\zeta^{(\gamma)}_i$.
Define also $N=\sum_{i,j} \zeta^{(r)}_i \zeta^{(c)}_j$ (which can
be seen as the normalization of the toric model when
$\zeta^{(\gamma)}$ is the unit vector, i.e., it is the vector with
all components equal to one).

We want to find three vectors $c=(c_1, \dots, c_I)$, $r=(r_1,
\dots, r_I)$, $d=(d_1, \dots, d_I)$, with $\sum r_i=\sum c_i=\sum
d_i=1$ and a scalar $0 \leq \alpha \leq 1$ such that

\begin{equation} \label{big}
\begin{aligned}
\frac{1}{N_T}
\begin{pmatrix}
\zeta^{(r)}_1\zeta^{(c)}_1\zeta^{(\gamma)}_1 & \zeta^{(r)}_1\zeta^{(c)}_2 & \zeta^{(r)}_1\zeta^{(c)}_3  & \dots &\zeta^{(r)}_1\zeta^{(c)}_I \\
\zeta^{(r)}_2\zeta^{(c)}_1  & \zeta^{(r)}_2\zeta^{(c)}_2\zeta^{(\gamma)}_2 & \zeta^{(r)}_2\zeta^{(c)}_3&  \dots &\zeta^{(r)}_2\zeta^{(c)}_I\\
\zeta^{(r)}_3\zeta^{(c)}_1 & \zeta^{(r)}_3\zeta^{(c)}_2 & \zeta^{(r)}_3\zeta^{(c)}_3\zeta^{(\gamma)}_3 & \dots & \zeta^{(r)}_3\zeta^{(c)}_I \\
\vdots & \vdots & \vdots & \vdots & \vdots \\
\zeta^{(r)}_I\zeta^{(c)}_1 & \zeta^{(r)}_I\zeta^{(c)}_2 &
\zeta^{(r)}_I\zeta^{(c)}_3 & \dots &
\zeta^{(r)}_I\zeta^{(c)}_I\zeta^{(\gamma)}_I
\end{pmatrix}=& \\
=\alpha
\begin{pmatrix}
r_1c_1 & r_1c_2 & r_1c_3  & \dots &r_1c_I \\
r_2c_1  & r_2c_2 & r_2c_3&  \dots &r_2c_I\\
r_3c_1 & r_3c_2 & r_3c_3 & \dots & r_3c_I \\
\vdots & \vdots & \vdots & \vdots & \vdots \\
r_Ic_1 & r_Ic_2 & r_Ic_3 & \dots & r_Ic_I
\end{pmatrix}
+(1-\alpha)
\begin{pmatrix}
d_1 & 0 & 0  & \dots & 0 \\
0 & d_2 & 0&  \dots & 0\\
0 & 0 & d_3 & \dots & 0 \\
\vdots & \vdots & \vdots & \vdots & \vdots \\
0 & 0 & 0 & \dots & d_I \\
\end{pmatrix} \, .
\end{aligned}
\end{equation}

We start studying the off-diagonal elements. Consider first the
case $N_T > N$. Thus we have
$\frac{\zeta^{(r)}_i\zeta^{(c)}_j}{N_T} <
\frac{\zeta^{(r)}_i\zeta^{(c)}_j}{N}$ and $\frac{N}{N_T} < 1$. In
this situation the only possible choice is given by
\begin{equation} \label{sol-sist}
\alpha=\frac{N}{N_T} \quad r_i=\frac{\zeta^{(r)}_i}{\sum
\zeta^{(r)}_i}, \quad c_j=\frac{\zeta^{(c)}_j}{\sum \zeta^{(c)}_j}
\, .
\end{equation}
In fact, recalling the definition of $N$, we have
\begin{equation} \label{extradiag}
\alpha r_i c_j=\frac{N}{N_T}\frac{\zeta^{(r)}_i}{\sum
\zeta^{(r)}_i}\frac{\zeta^{(c)}_j}{\sum
\zeta^{(c)}_j}=\frac{N}{N_T}\frac{\zeta^{(r)}_i\zeta^{(c)}_j}{N}=\frac{\zeta^{(r)}_i\zeta^{(c)}_j}{N_T}
\end{equation}
for all $i$, $j$ with $i \not=j$. Taking the log-probabilities, we
obtain a linear system. It is easy to prove, as in Chapter 6 of
\cite{pistone|riccomagno|wynn:01}, that the rank of this system is
equal to $(2I-1)$. Hence, considering the normalizing equations
for $r$ and $c$, we see that the solution in \eqref{sol-sist} is
unique.

Let us consider the generic equation of the $i$-th diagonal
element:
\begin{equation*}
\zeta^{(r)}_i\zeta^{(c)}_i\zeta^{(\gamma)}_i=\alpha
r_ic_i+(1-\alpha)d_i \, .
\end{equation*}
After substituting the previous values for $r_i$, $c_i$ and
$\alpha$ we get
\begin{equation*}
\zeta^{(r)}_i\zeta^{(c)}_i\zeta^{(\gamma)}_i=\frac{N}{N_T}\frac{\zeta^{(r)}_i\zeta^{(c)}_i}{N}+\frac{N_T-N}{N_T}d_i
\, .
\end{equation*}
As we consider matrices in $\Delta_{>0}$, the quantity
$\zeta^{(r)}_i\zeta^{(c)}_i$ is different from zero. Therefore,
after multiplying for $N_T$ and dividing by
$\zeta^{(r)}_i\zeta^{(c)}_i$ we obtain
\begin{equation*}
\zeta^{(\gamma)}_i=1+\frac{N_T-N}{\zeta^{(r)}_i\zeta^{(c)}_i}d_i
\end{equation*}
that is
\begin{equation*}
d_i=\frac{\zeta^{(r)}_i\zeta^{(c)}_i}{N_T-N}(\zeta^{(\gamma)}_i-1)
\end{equation*}
Thus we see that the $P \in {\mathcal{M}}_1\setminus
{\mathcal{M}}_2$ when $N_T> N$ and there exists at least an index
$i$ such that $\zeta^{(\gamma)}_i<1$.

When $N_T=N$, from Equations \eqref{extradiag} we obtain
$\alpha=1$. Therefore in Equation \eqref{big} the matrix on the
right hand side has rank 1, and this implies that $P \in {\mathcal
M}_2$ if and only if $\zeta^{(\gamma)}_i=1$ for all $i$.

Consider now the case $N_T < N$. Hence we have
$\frac{\zeta^{(r)}_i\zeta^{(c)}_j}{N_T}>
\frac{\zeta^{(r)}_i\zeta^{(c)}_j}{N}$ and $\frac{N}{N_T}> 1$.
Again the only possible choice for the off-diagonal elements would
be given by
\begin{equation*}
\alpha=\frac{N}{N_T} \quad r_i=\frac{\zeta^{(r)}_i}{\sum
\zeta^{(r)}_i}, \quad c_i=\frac{\zeta^{(c)}_i}{\sum \zeta^{(c)}_i}
\end{equation*}
but in this case $\alpha=\frac{N}{N_T}>1$. Thus we conclude that
all $P \in {\mathcal{M}}_1$ with $N_T < N$ are in
${\mathcal{M}}_1\setminus {\mathcal{M}}_2$. This leads to the
following result.

\begin{theorem}
Let $P \in {\mathcal M}_1 \cap \Delta_{>0}$ be a strictly positive
probability table given by the vectors
$\zeta^{(r)}=(\zeta^{(r)}_1, \dots, \zeta^{(r)}_I)$,
$\zeta^{(c)}=(\zeta^{(c)}_1, \dots, \zeta^{(c)}_I)$ and
$\zeta^{(\gamma)}=(\zeta^{(\gamma)}_1, \dots,
\zeta^{(\gamma)}_I)$. Define $N_T=\sum_{i\not=
j}\zeta^{(r)}_i\zeta^{(c)}_j+\sum_{i=j}\zeta^{(r)}_i\zeta^{(c)}_j\zeta^{(\gamma)}_i$
and $N=\sum_{i,j} \zeta^{(r)}_i\zeta^{(c)}_j$. Then $P \in
{\mathcal M}_1 \setminus {\mathcal M}_2$ if one of the following
situations holds:
\begin{itemize}
\item[(i)] $N_T < N$;

\item[(ii)] $N_T = N$ and there exists at least an index $i$ such
that $\zeta^{(\gamma)}_i\not=1$;

\item[(iii)] $N_T> N$ and there exists at least an index $i$ such
that $\zeta^{(\gamma)}_i<1$.
\end{itemize}
\end{theorem}

We conclude this section with a result on the sufficient
statistics for the models ${\mathcal M}_1$ and ${\mathcal M}_2$.

\begin{proposition}
For an independent sample of size $n$, the models ${\mathcal M}_1$
and ${\mathcal M}_2$ have the same sufficient statistic.
\end{proposition}
\begin{proof}
In fact, let $f=(f_{i,j})$ be the table of counts for the sample.
The likelihood function for the model in toric form is
\begin{equation*}
L_1(\zeta^{(r)},\zeta^{(c)},\zeta^{(\gamma)};f) = \prod_{i,j}
p_{i,j}^{f_{i,j}} = \prod_{i \ne j}
(\zeta^{(r)}_i\zeta^{(c)}_j)^{f_{i,j}} \prod_{i}
(\zeta^{(r)}_i\zeta^{(c)}_i\zeta^{(\gamma)}_i)^{f_{i,i}} =
\end{equation*}
\begin{equation*}
=  \prod_{i} (\zeta^{(r)})^{f_{i,+}} \prod_{j}
(\zeta^{(c)}_j)^{f_{+,j}} \prod_{i} (\zeta^{(\gamma)}_i)^{f_{i,i}}
\, ,
\end{equation*}
where $f_{i,+}$ and $f_{+,j}$ are the row and column marginal
totals, respectively. This proves that the marginal totals
together with the counts on the main diagonal are a sufficient
statistic. With the same statistic we can also write the
likelihood under the mixture model ${\mathcal M}_2$:
\begin{equation*}
L_2(r,c,d,\alpha;f) = \prod_{i,j} p_{i,j}^{f_{i,j}} = \prod_{i \ne
j} (\alpha r_i c_j)^{f_{i,j}} \prod_{i} (\alpha r_i c_i +
(1-\alpha) d_i)^{f_{i,i}} =
\end{equation*}
\begin{equation*}
= \alpha^{(n-\sum_{i} f_{i,i})} \prod_{i} r_i^{(f_{i,+}-f_{i,i})}
\prod_{j} c_j^{(f_{+,j}-f_{j,j})} \prod_{i} (\alpha r_i c_i +
(1-\alpha) d_i)^{f_{i,i}} \, .
\end{equation*}
\end{proof}

\section{A geometric description of the diagonal-effect models} \label{geo-desc}

In this section, we try to describe the models we studied using
some geometric flavor. This analysis will also shed some light on
the elements in $\mathcal{M}_1\setminus\mathcal{M}_2$. We use very
basic and classic geometric ideas and facts. As references, we
suggest \cite{harris:92} and \cite{hartshorne:77}.

We start with the model ${\mathcal M}_1$. The basic object we need
is the variety $V$ describing all $I\times I$ matrices having rank
at most one. When we fix $\zeta^{(\gamma)}_i=1,i=1,\ldots,I$ the
parametrization in \eqref{toric-def-1} and \eqref{toric-def-2} is
just describing $V$. Hence, fixing values for all the
$\zeta^{(c)}_i$'s and the $\zeta^{(r)}_i$'s and setting
$\zeta^{(\gamma)}_j=1, j=1,\ldots,I$ we obtain a point $M\in V$.
Now, if we let $\zeta^{(\gamma)}_l$ to vary we are describing a
line passing through $M$ and moving in the direction of the vector
$(0,\ldots,1,\ldots,0)$, where the only non zero coordinate is the
$(l,l)$-th; the set of all these lines is a cylinder. Now we set
$\zeta^{(\gamma)}_l=a\zeta$ and $\zeta^{(\gamma)}_m=b\zeta$ for
fixed reals $a$ and $b$. When we let $\zeta$ vary, we are now
describing a cylinder with directrix parallel to the line of
equations $bp_{l,l}-ap_{m,m}$, $p_{i,j}=0$ for $(i,j)\neq
(l,l),(m,m)$. The same argument can be repeated fixing linear
relations among the diagonal elements.  In conclusion, we can
describe $\mathcal{M}_1$ as the intersection of the simplex with
the union of cylinders having base $V$ and directrix parallel to
the directions given by diagonal elements.

We now use the join of two varieties, i.e. the closure of the set
of all the lines joining a point of any variety with any point of
another variety. In order to do this, we also need to consider $W$
the variety of diagonal matrices. Then $\mathcal{M}_2$ is the
union of the segment joining a point of $V\cap\Delta$ with a point
of $W\cap\Delta$, i.e. a subvariety of the join of $V$ and $W$.
Each of this segment lies on a line contained in one of the
cylinder we used to construct $\mathcal{M}_1$. Hence we get again
the inclusion $\mathcal{M}_2\subset\mathcal{M}_1$ in $\Delta$.

\section{Common-diagonal-effect models} \label{commdiagsect}

A different version of the diagonal-effect models are the
so-called common-diagonal-effect models. The definitions are as in
the models above but:
\begin{itemize}
\item The vector $\zeta^{(\gamma)}$ is constant in the toric model
definition;

\item The matrix $D$ is ${\mathrm{diag}}(\frac 1 I, \ldots, \frac
1 I)$ in the mixture model definition.
\end{itemize}

This kind of models is much more complicated than the models in
Section \ref{diageffsect}. Just to have a first look at these
models, we note that for $I=3$ the diagonal-effect models have
only one invariant. For the common-diagonal-effect models, we have
computed the invariants with CoCoA, see \cite{cocoa}, for $I=3$
and we have obtained the following lists of invariants.

For the toric model we obtain $9$ binomials:
\begin{equation*}
p_{1,2}p_{2,3}p_{3,1} - p_{1,3}p_{2,1}p_{3,2}\, ,
\end{equation*}
\begin{equation*}
p_{1,3}p_{2,2}p_{3,1} - p_{1,1}p_{2,3}p_{3,2} \, ,
\end{equation*}
\begin{equation*}
-p_{1,1}p_{2,3}p_{3,2} + p_{1,2}p_{2,1}p_{3,3} \, ,
\end{equation*}
\begin{equation*}
-p_{2,2}p_{2,3}p_{3,1}^2 + p_{2,1}^2p_{3,2}p_{3,3} \, ,
\end{equation*}
\begin{equation*}
p_{1,2}p_{2,2}p_{3,1}^2 - p_{1,1}p_{2,1}p_{3,2}^2 \, ,
\end{equation*}
\begin{equation*}
-p_{1,1}p_{1,3}p_{3,2}^2 + p_{1,2}^2p_{3,1}p_{3,3} \, ,
\end{equation*}
\begin{equation*}
-p_{1,3}^2p_{2,2}p_{3,2} + p_{1,2}^2p_{2,3}p_{3,3} \, ,
\end{equation*}
\begin{equation*}
-p_{1,1}p_{2,3}^2p_{3,1} + p_{1,3}p_{2,1}^2p_{3,3} \, ,
\end{equation*}
\begin{equation*}
p_{1,3}^2p_{2,1}p_{2,2} - p_{1,1}p_{1,2}p_{2,3}^2 \, .
\end{equation*}

For the mixture model we obtain:
\begin{itemize}
\item $1$ binomial
\begin{equation*}
p_{1,2}p_{2,3}p_{3,1} - p_{1,3}p_{2,1}p_{3,2} \, ;
\end{equation*}
\item $12$ polynomials with $4$ terms
\begin{equation*}
p_{1,3}p_{2,1}p_{2,2} - p_{1,2}p_{2,1}p_{2,3} +
p_{1,3}p_{2,3}p_{3,1} - p_{1,3}p_{2,1}p_{3,3} \, ,
\end{equation*}
\begin{equation*}
-p_{1,2}p_{1,3}p_{2,2} + p_{1,2}^2p_{2,3} - p_{1,3}^2p_{3,2} +
p_{1,2}p_{1,3}p_{3,3} \, ,
\end{equation*}
\begin{equation*}
 p_{1,3}p_{2,1}p_{3,1} -
p_{1,1}p_{2,3}p_{3,1} + p_{2,2}p_{2,3}p_{3,1} -
p_{2,1}p_{2,3}p_{3,2} \, ,
\end{equation*}
\begin{equation*}
 p_{1,2}p_{1,3}p_{3,1} -
p_{1,1}p_{1,3}p_{3,2} + p_{1,3}p_{2,2}p_{3,2} -
p_{1,2}p_{2,3}p_{3,2} \, ,
\end{equation*}
\begin{equation*}
 p_{1,3}p_{2,1}^2 -
p_{1,1}p_{2,1}p_{2,3} - p_{2,3}^2p_{3,1} + p_{2,1}p_{2,3}p_{3,3}
\, ,
\end{equation*}
\begin{equation*}
 p_{1,3}^2p_{2,1} - p_{1,1}p_{1,3}p_{2,3} +
p_{1,3}p_{2,2}p_{2,3} - p_{1,2}p_{2,3}^2 \, ,
\end{equation*}
\begin{equation*}
p_{1,2}p_{1,3}p_{2,1} - p_{1,1}p_{1,2}p_{2,3} -
p_{1,3}p_{2,3}p_{3,2} + p_{1,2}p_{2,3}p_{3,3} \, ,
\end{equation*}
\begin{equation*}
-p_{2,1}p_{2,2}p_{3,1} - p_{2,3}p_{3,1}^2 + p_{2,1}^2p_{3,2} +
p_{2,1}p_{3,1}p_{3,3} \, ,
\end{equation*}
\begin{equation*}
-p_{1,2}p_{2,2}p_{3,1} + p_{1,2}p_{2,1}p_{3,2} -
p_{1,3}p_{3,1}p_{3,2} + p_{1,2}p_{3,1}p_{3,3} \, ,
\end{equation*}
\begin{equation*}
p_{1,2}p_{3,1}^2 - p_{1,1}p_{3,1}p_{3,2} - p_{2,2}p_{3,1}p_{3,2} -
p_{2,1}p_{3,2}^2 \, ,
\end{equation*}
\begin{equation*}
p_{1,2}p_{2,1}p_{3,1} - p_{1,1}p_{2,1}p_{3,2} -
p_{2,3}p_{3,1}p_{3,2} + p_{2,1}p_{3,2}p_{3,3} \, ,
\end{equation*}
\begin{equation*}
p_{1,2}^2p_{3,1} - p_{1,1}p_{1,2}p_{3,2} - p_{1,3}p_{3,2}^2 +
p_{1,2}p_{3,2}p_{3,3} \, ;
\end{equation*}
\item $6$ polynomials with $8$ terms
\begin{multline*}
p_{1,1}p_{1,3}p_{2,2} - p_{1,3}p_{2,2}^2 - p_{1,1}p_{1,2}p_{2,3} +
p_{1,2}p_{2,2}p_{2,3} + \\ + p_{1,3}^2p_{3,1} -
p_{1,3}p_{2,3}p_{3,2} - p_{1,1}p_{1,3}p_{3,3} +
p_{1,3}p_{2,2}p_{3,3} \, ,
\end{multline*}
\begin{multline*}
p_{1,1}p_{1,3}p_{2,1} - p_{1,1}^2p_{2,3} - p_{1,2}p_{2,1}p_{2,3} +
p_{1,1}p_{2,2}p_{2,3} + \\ + p_{2,3}^2p_{3,2} -
p_{1,3}p_{2,1}p_{3,3} + p_{1,1}p_{2,3}p_{3,3} -
p_{2,2}p_{2,3}p_{3,3} \, ,
\end{multline*}
\begin{multline*}
-p_{1,1}p_{2,2}p_{3,1} + p_{2,2}^2p_{3,1} - p_{1,3}p_{3,1}^2 +
p_{1,1}p_{2,1}p_{3,2} + \\ - p_{2,1}p_{2,2}p_{3,2} +
p_{2,3}p_{3,1}p_{3,2} + p_{1,1}p_{3,1}p_{3,3} -
p_{2,2}p_{3,1}p_{3,3} \, ,
\end{multline*}
\begin{multline*}
p_{1,1}p_{1,2}p_{3,1} - p_{1,1}^2p_{3,2} - p_{1,2}p_{2,1}p_{3,2} +
p_{1,1}p_{2,2}p_{3,2} + \\ + p_{2,3}p_{3,2}^2 -
p_{1,2}p_{3,1}p_{3,3} + p_{1,1}p_{3,2}p_{3,3} -
p_{2,2}p_{3,2}p_{3,3} \, ,
\end{multline*}
\begin{multline*}
 p_{1,2}p_{2,1}^2 - p_{1,1}p_{2,1}p_{2,2} - p_{1,1}p_{2,3}p_{3,1} -
p_{2,1}p_{2,3}p_{3,2} + \\ + p_{1,1}p_{2,1}p_{3,3} +
p_{2,1}p_{2,2}p_{3,3} + p_{2,3}p_{3,1}p_{3,3} - p_{2,1}p_{3,3}^2
\, ,
\end{multline*}
\begin{multline*}
p_{1,2}^2p_{2,1} - p_{1,1}p_{1,2}p_{2,2} - p_{1,1}p_{1,3}p_{3,2} -
p_{1,2}p_{2,3}p_{3,2} + \\ + p_{1,1}p_{1,2}p_{3,3} +
p_{1,2}p_{2,2}p_{3,3} + p_{1,3}p_{3,2}p_{3,3} - p_{1,2}p_{3,3}^2
\, ;
\end{multline*}
\item $1$ polynomial with $12$ terms
\begin{multline*}
p_{1,1}p_{1,2}p_{2,1} - p_{1,1}^2p_{2,2} - p_{1,2}p_{2,1}p_{2,2} +
p_{1,1}p_{2,2}^2 + \\ - p_{1,1}p_{1,3}p_{3,1} +
p_{2,2}p_{2,3}p_{3,2} + p_{1,1}^2p_{3,3} - p_{2,2}^2p_{3,3} + \\ +
p_{1,3}p_{3,1}p_{3,3} - p_{2,3}p_{3,2}p_{3,3} - p_{1,1}p_{3,3}^2 +
p_{2,2}p_{3,3}^2 \, .
\end{multline*}
\end{itemize}

In the case of toric models, the invariants can be characterized
theoretically. In fact, also in this case a Markov basis is known.
In \cite{hara|takemura|yoshida:08} it is shown that a Markov basis
for this toric model is formed by $6$ different types of moves. We
need the 2 types of moves for the diagonal-effect model plus the
moves below:
\begin{itemize}
\item The degree $3$ moves of the form:
\begin{center}
\begin{tabular}{cccc}
    & $i$ & $i'$ & $i''$ \\
$i$ & $+1$ & $0$ & $-1$ \\
$i'$ & $0$ & $-1$ & $+1$ \\
$i''$ & $-1$ & $+1$ & $0$
\end{tabular} \ \ \
\begin{tabular}{cccc}
    & $i$ & $i'$ & $i''$ \\
$i$ & $+1$ & $-1$ & $0$ \\
$i'$ & $-1$ & $0$ & $+1$ \\
$i''$ & $0$ & $+1$ & $-1$
\end{tabular} \ \ \
\begin{tabular}{cccc}
    & $i$ & $i'$ & $i''$ \\
$i$ & $0$ & $-1$ & $+1$ \\
$i'$ & $-1$ & $+1$ & $0$ \\
$i''$ & $+1$ & $0$ & $-1$
\end{tabular}
\end{center}
with $i$, $i'$, $i''$ all distinct, for $I \geq 3$;

\item The degree $3$ moves of the form:
\begin{center}
\begin{tabular}{cccc}
    & $i$ & $i'$ & $j$  \\
$i$ & $+1$ & $0$ & $-1$ \\
$i'$ & $0$ & $-1$ & $+1$ \\
$j'$ & $-1$ & $+1$ & $0$
\end{tabular}
\end{center}
with $i$, $i'$, $j$, $j'$ all distinct, for $I \geq 4$.

\item The degree $4$ moves of the form:
\begin{center}
\begin{tabular}{cccc}
    & $i$ & $i'$ & $j$  \\
$i$ & $+1$ & $+1$ & $-2$ \\
$i'$ & $-1$ & $-1$ & $+2$
\end{tabular}
\end{center}
with $i$, $i'$, $j$ all distinct, for $I \geq 3$, and their
transposed.

\item The degree $4$ moves of the form:
\begin{center}
\begin{tabular}{ccccc}
    & $i$ & $i'$ & $j$ & $j'$ \\
$i$ & $+1$ & $+1$ & $-1$ & $-1$ \\
$i'$ & $-1$ & $-1$ & $+1$ & $+1$
\end{tabular}
\end{center}
with $i$, $i'$, $j$, $j'$ all distinct, for $I \geq 4$, and their
transposed.
\end{itemize}

Therefore, as in Proposition \ref{appl_DS}, we can easily derive
the invariants. We do not write explicitly the analog of
Proposition \ref{appl_DS} for common-diagonal-effect models in
order to save space.

The study of the common-diagonal-effect models in mixture form is
much more complicated. In fact, notice that in the computations
above, the mixture model present invariants which are not
binomials. However, some partial results can be stated.

\begin{theorem} \label{inv-common}
\begin{itemize}
\item[(a)] For $i$, $j$, $k$, $l$ all distinct we define
\begin{equation*}
b_{ijkl}=p_{i,j}p_{k,l} - p_{i,l}p_{k,j}  \, ;
\end{equation*}
\item[(b)] For $i$, $j$, $k$, all distinct we define
\begin{equation*}
t_{ijk}=p_{i,j}p_{j,k}p_{k,i} - p_{i,k}p_{k,j}p_{j,i}  \, ;
\end{equation*}
\item[(c)]For $(i,j)$ and $(k,l)$ two distinct pairs in $\{1,
\ldots, I\}$ with $i\not= j$, and $k\not= l$ and $m\in \{1, \ldots
, I\} \setminus \{i,j\}$ and $n \in \{1, \ldots, I\} \setminus
\{k,l\}$ with $m\not=n$ we define
\begin{equation*}
f_{ijklmn}=p_{i,j}p_{k,l}p_{n,n}-p_{i,j}p_{n,l}p_{k,n}-p_{i,j}p_{k,l}p_{m,m}+p_{k,l}p_{m,j}p_{i,m}
\, ;
\end{equation*}
\item[(d)] for two distinct indices $i$ and $j$ in $\{1, \ldots,
I\}$ and for $k \in \{1, \ldots, I\} \setminus \{i,j\}$ we define
\begin{equation*}
\begin{split}
g_{ijk}= & p_{i,j}p_{i,i}p_{k,k}+p_{i,j}p_{j,j}p_{k,k}-p_{i,j}p_{i,i}p_{j,j}+p_{i,j}p_{k,k}p_{k,k}+\\
&+p_{k,k}p_{i,k}p_{k,j}-p_{i,i}p_{i,k}p_{k,j}+p_{i,j}^2p_{j,i}-p_{i,j}p_{k,j}p_{j,k}
\, ;
\end{split}
\end{equation*}
\item[(e)] For $i$, $j$, $k$, all distinct we define
\begin{equation*}
\begin{split}
h_{ijk}= & p_{i,i}p_{j,j}^2+p_{i,i}^2p_{k,k}+p_{j,j}p_{k,k}^2-p_{i,i}^2p_{j,j}-p_{j,j}^2p_{k,k}-p_{i,i}p_{k,k}^2+ p_{i,i}p_{i,j}p_{j,i}+\\
 &-p_{i,i}p_{i,k}p_{k,i}+p_{j,j}p_{j,k}p_{k,j}-p_{j,j}p_{j,i}p_{i,j}+p_{k,k}p_{k,i}p_{i,k}-p_{k,k}p_{k,j}p_{j,k}  \, .
\end{split}
\end{equation*}
\end{itemize}
Then the previous polynomials are invariants for the
common-diagonal-effect models in mixture form.
\end{theorem}

\begin{proof}
Cases (a) and (b) follow from Proposition \ref{appl_DS} since the
off-diagonal elements of the probability table are described, up
to scalar, in the same monomial form as for the elements of
${\mathcal{M}}_1$.

For case (c), consider the term $g_1=p_{i,j}p_{k,l}p_{n,n}$ in
$f_{ijklmn}$. This gives two monomials: $\alpha^3
r_ic_jr_kc_lr_nc_n$ and $\alpha^2 r_ic_jr_kc_l(1-\alpha)d$, where
$d=1/I$. The term $-g_2=-p_{i,j}p_{n,l}p_{k,n}$ of $f_{ijklmn}$
cancels the first monomial of $g_1$. In fact
$-p_{i,j}p_{n,l}p_{k,n}=\alpha^3 r_ic_jr_nc_lr_kc_n$. Since in
$g_2$ there are not diagonal variables, we need another term in
order to cancel the second monomial of $g_1$. Thus we subtract, to
$g_1-g_2$, a term of the form $g_3=p_{i,j}p_{k,l}p_{m,m}$ which
gives the monomials $-\alpha^2 r_ic_jr_kc_l (1-\alpha)d$ and
$-\alpha^3 r_ic_jr_kc_lr_mc_m$. To cancel this last monomial it is
enough to add the term
$g_4=p_{k,l}p_{m,j}p_{i,m}=\alpha^3r_kc_lr_mc_jr_ic_m$. Thus
$f_{ijklmn}=g_1-g_2-g_3+g_4$ vanishes on the entries of a
probability table of the mixture model with common diagonal
effect.

For case (d), consider first the terms with pairs of variables on
the diagonal.
\begin{equation*}
\begin{split}
p_{i,j}p_{i,i}p_{k,k}=& \alpha^3r_i^2r_kc_ic_jc_k+\alpha^2r_i^2c_ic_jd-\alpha^3r_i^2c_ic_jd+\boxed{\alpha^2r_ir_kc_jc_kd}+\\
& +\alpha
r_ic_jd^2-2\alpha^2r_ic_jd^2-\alpha^3r_ir_kc_jc_kd+\alpha^3
r_ic_jd^2  \, ;
\end{split}
\end{equation*}

\begin{equation*}
\begin{split}
p_{i,j}p_{j,j}p_{k,k}=& \alpha^3r_ir_jr_kc_j^2c_k+\alpha^2r_ir_jc_j^2d-\alpha^3r_ir_jc_j^2d+\alpha^2r_ir_kc_jc_kd+\\
& +\alpha
r_ic_jd^2-2\alpha^2r_ic_jd^2-\boxed{\alpha^3r_ir_kc_jc_kd}+\alpha^3
r_ic_jd^2  \, ;
\end{split}
\end{equation*}

\begin{equation*}
\begin{split}
p_{i,j}p_{i,i}p_{j,j}=& \alpha^3r_i^2r_jc_ic_j^2+\alpha^2r_i^2c_ic_jd-\alpha^3r_i^2c_ic_jd+\alpha^2r_ir_jc_j^2d+\\
& +\alpha
r_ic_jd^2-2\alpha^2r_ic_jd^2-\alpha^3r_ir_jc_j^2d+\alpha^3
r_ic_jd^2 \, ;
\end{split}
\end{equation*}

\begin{equation*}
\begin{split}
p_{i,j}p_{k,k}^2=& \alpha^3r_ir_k^2c_jc_k^2+2\alpha^2r_ir_kc_jc_kd-\boxed{2\alpha^3r_ir_kc_jc_kd}+\alpha r_ic_jd^2+\\
&-2\alpha^2r_ic_jd^2+\alpha^3 r_ic_jd^2 \, .
\end{split}
\end{equation*}

It is easy to see that while some terms, such as $\alpha^3
r_ic_jd^2$, are simply cancelled considering the difference of two
monomials, other terms, such as the boxed ones, appear in
different monomials. However, they appear with the appropriate
coefficients and considering
$p_{i,j}p_{i,i}p_{k,k}+p_{i,j}p_{j,j}p_{k,k}-p_{i,j}p_{i,i}p_{j,j}-p_{i,j}p_{k,k}^2$
we cancel most of them. In fact we obtain
\begin{equation*}
\alpha^3r_i^2r_kc_ic_jc_k-\alpha^3r_ir_k^2c_jc_k^2-\alpha^3r_i^2r_jc_j^2c_i+\alpha^3r_ir_jr_kc_j^2c_k
\, .
\end{equation*}

The only way to cancel the term $-\alpha^3r_ir_k^2c_jc_k^2$ is to
add the monomial
$p_{i,k}p_{k,j}p_{k,k}=\alpha^3r_ir_k^2c_jc_k^2+\alpha^2r_ir_kc_jc_kd-\alpha^3
r_ir_kc_jc_k d$. However this monomial adds two more terms that
can be cancelled by using another monomial with a variable in the
diagonal, that is
$p_{i,i}p_{i,k}p_{k,j}=\alpha^3r_i^2r_kc_jc_jc_k+\alpha^2r_ir_kc_jc_kd-\alpha^3
r_ir_kc_jc_k d$. After that, the only two missing terms are
$-\alpha^3r_i^2r_jc_j^2c_i+\alpha^3r_ir_jr_kc_j^2c_k$ which can be
cancelled by adding $p_{i,j}^2p_{j,i}-p_{i,j}p_{k,j}p_{j,k}$.

For the case (e), we omit the complete details of the proof. One
has to proceed as in cases (c) and (d) considering separately
$p_{i,i}p_{j,j}^2+p_{i,i}^2p_{k,k}+p_{j,j}p_{k,k}^2-p_{i,i}^2p_{j,j}-p_{j,j}^2p_{k,k}-p_{i,i}p_{k,k}^2$
and the contributions of
$p_{i,i}p_{i,j}p_{j,i}-p_{i,i}p_{i,k}p_{k,i}$,
$p_{j,j}p_{j,k}p_{k,j}-p_{j,j}p_{j,i}p_{i,j}$ and
$p_{k,k}p_{k,i}p_{i,k}-p_{k,k}p_{k,j}p_{j,k}$.
\end{proof}

With some computations with CoCoA, we have found that the
polynomials defined in Theorem \ref{inv-common} define the model
${\mathcal M}_2$ for $I=3,4,5$. We conjecture that this fact is
true in general.

\bibliographystyle{decsci}
\bibliography{tuttopm}

\end{document}